\documentclass[a4paper,10pt,reqno]{amsart}
\usepackage{xcolor}
\usepackage{tabularx}
\usepackage{ltablex}
\usepackage{caption}

\usepackage{tikz}
\usetikzlibrary{shapes.geometric}
\usetikzlibrary {shapes.misc}
\usetikzlibrary{positioning}

\usepackage[backref,hypertexnames=false,linktoc=all]{hyperref}

\newcommand{\hilb}{\text{-Hilb}}
\newcommand{\Conj}{\operatorname{Conj}}
\newcommand{\Exc}{\operatorname{Exc}}

\usepackage{color}

\usepackage{comment}
\usepackage{enumerate}

\usepackage{latexsym}
\usepackage{amscd}
\usepackage{amsmath}
\usepackage{amssymb}
\usepackage{mathrsfs}

\usepackage{array}
\usepackage{amsthm}
\usepackage{float}
\usepackage{graphicx}
\usepackage{subfigure}
\usepackage{pst-all}

\theoremstyle{plain}
\newtheorem{theorem}{Theorem}[section]
\newtheorem{corollary}[theorem]{Corollary}
\newtheorem{lemma}[theorem]{Lemma}
\newtheorem{definition-lemma}[theorem]{Definition-Lemma}

\newtheorem{proposition}[theorem]{Proposition}
\newtheorem{conjecture}[theorem]{Conjecture}

\theoremstyle{definition}

\newtheorem{remark}[theorem]{Remark}

\newtheorem{example}[theorem]{Example}

\numberwithin{equation}{section}
\numberwithin{figure}{section}
\numberwithin{table}{section}

\usepackage{prettyref}

\newrefformat{th}{Theorem~\ref{#1}}
\newrefformat{cr}{Corollary~\ref{#1}}
\newrefformat{lm}{Lemma~\ref{#1}}
\newrefformat{dl}{Definition-Lemma~\ref{#1}}
\newrefformat{cl}{Claim~\ref{#1}}
\newrefformat{sl}{Sublemma~\ref{#1}}
\newrefformat{pr}{Proposition~\ref{#1}}
\newrefformat{cj}{Conjecture~\ref{#1}}
\newrefformat{st}{Step~\ref{#1}}
\newrefformat{sc}{Section~\ref{#1}}
\newrefformat{df}{Definition~\ref{#1}}
\newrefformat{rm}{Remark~\ref{#1}}
\newrefformat{q}{Question~\ref{#1}}
\newrefformat{pb}{Problem~\ref{#1}}
\newrefformat{cd}{Condition~\ref{#1}}
\newrefformat{eg}{Example~\ref{#1}}
\newrefformat{he}{Heore~\ref{#1}}
\newrefformat{fg}{Figure~\ref{#1}}
\newrefformat{tb}{Table~\ref{#1}}

\newdimen\argwidth
\def\db[#1\db]{%
 \setbox0=\hbox{$#1$}\argwidth=\wd0
 \setbox0=\hbox{$\left[\box0\right]$}
  \advance\argwidth by -\wd0
 \left[\kern.3\argwidth\box0 \kern.3\argwidth\right]}

\newcommand{\diag}{\operatorname{diag}}

\newcommand{\id}{\operatorname{id}}

\newcommand{\GL}{GL}

\newcommand{\SL}{SL}

\newcommand{\Irrep}{\operatorname{Irrep}}

\newcommand{\Hom}{\operatorname{Hom}}

\newcommand{\Spec}{\operatorname{Spec}}

\newcommand{\Sing}{\operatorname{Sing}}

\newcommand{\bA}{\ensuremath{\mathbb{A}}}

\newcommand{\bC}{\ensuremath{\mathbb{C}}}

\newcommand{\bP}{\ensuremath{\mathbb{P}}}
\newcommand{\bQ}{\ensuremath{\mathbb{Q}}}
\newcommand{\bR}{\ensuremath{\mathbb{R}}}

\newcommand{\bZ}{\ensuremath{\mathbb{Z}}}

\newcommand{\scC}{\ensuremath{\mathcal{C}}}
\newcommand{\scD}{\ensuremath{\mathcal{D}}}

\newcommand{\scM}{\ensuremath{\mathcal{M}}}

\newcommand{\scO}{\ensuremath{\mathcal{O}}}

\newcommand{\scU}{\ensuremath{\mathcal{U}}}

\newcommand{\scX}{\ensuremath{\mathcal{X}}}
\newcommand{\scY}{\ensuremath{\mathcal{Y}}}

\newcommand{\Dtilde}{{\widetilde{D}}}

\newcommand{\Xtilde}{{\widetilde{X}}}
\newcommand{\Ytilde}{{\widetilde{Y}}}

\newcommand{\Dbar}{{\overline{D}}}

\newcommand{\Rbar}{{\overline{R}}}
\newcommand{\Sbar}{{\overline{S}}}

\newcommand{\Dhat}{\widehat{D}}

\makeatletter
\newcommand{\xleftrightarrows}[2][]{\mathrel{
 \raise.40ex\hbox{$
       \ext@arrow 3095\leftarrowfill@{\phantom{#1}}{#2}$}
 \setbox0=\hbox{$\ext@arrow 0359\rightarrowfill@{#1}{\phantom{#2}}$}
 \kern-\wd0 \lower.40ex\box0}}

\newcommand{\xrightleftarrows}[2][]{\mathrel{
 \raise.40ex\hbox{$\ext@arrow 3095\rightarrowfill@{\phantom{#1}}{#2}$}
 \setbox0=\hbox{$\ext@arrow 0359\leftarrowfill@{#1}{\phantom{#2}}$}
 \kern-\wd0 \lower.40ex\box0}}
\newcommand{\xleftrightarrow}[2][]{
     \ext@arrow 0055{\leftrightarrowfill@}{#1}{#2}
}
\def\leftrightarrowfill@{
 \arrowfill@\leftarrow\relbar\rightarrow
}  
\makeatother

\title[Derived McKay correspondence for real reflection groups]{Derived McKay correspondence for real reflection groups of rank three}

\author[A. Ishii]{Akira Ishii}
\address{Graduate School of Mathematics, Nagoya University, Furocho, Chikusa-ku, Nagoya, 464-8602, Japan}
\email{akira141@math.nagoya-u.ac.jp}

\author[S. Nimura]{Shu Nimura}
\address{Graduate School of Mathematics, Nagoya University, Furocho, Chikusa-ku, Nagoya, 464-8602, Japan}
\email{shu.nimura.c6@math.nagoya-u.ac.jp}

\begin{document}

\begin{abstract}
We describe the derived McKay correspondence for real reflection groups of rank $3$
in terms of a maximal resolution of the logarithmic pair consisting of the quotient variety and the discriminant divisor
with coefficient $\frac{1}{2}$.
As an application, we verify a conjecture by Polishchuk and Van den Bergh
on the existence of a certain semiorthgonal decomposition of the equivariant derived category
into the derived categories of affine spaces for
any real reflection group of rank $3$.
\end{abstract}

\maketitle

%
%
%

%
%
\section{Introduction}\label{sc:Introduction}
The derived McKay correspondence, established for $G\subset \SL(2, \bC)$ \cite{Kapranov-Vasserot} and $G\subset \SL(3, \bC)$ \cite{BKR},
states that there is an equivalence $D^G(\bA^n) \cong D(Y)$ between the bounded derived category $D^G(\bA^n)$ of
$G$-equivariant coherent sheaves on $\bA^n$ and the bounded derived category $D(Y)$ of coherent sheaves
on a crepant resolution $Y$ of $\bA^n/G$.
In the case $G\subset \GL(n, \bC)$ and $G\not\subset\SL(n, \bC)$ with $n=2, 3$, similar correspondences have been described in terms of semiorthogonal decompositions
\cite{Ishii-Ueda} \cite{kawamata_toric_III} \cite{kawamata_GL3}.
To describe the McKay correspondence for subgroups of $\GL(3, \bC)$ in \cite{kawamata_GL3}, Kawamata defined and constructed a {\it maximal $\bQ$-factorial terminalization} $Y$ with only quotient singularities, and obtained a semiorthogonal decomposition of $D^G(\bA^3)$ in which the derived category $D(\Ytilde)$ of the canonical stack $\Ytilde$ of $Y$ appears as a component \cite{kawamata_GL3}.

A feature of Kawamata's results  \cite{kawamata_toric_III} \cite{kawamata_GL3}
is that his construction is valid even when $G$ contains {\it pseudoreflections},
where an element of $\GL(n, \bC)$ is called a pseudoreflection if its fixed point set is a hyperplane.
 In such a case,  one should consider not the quotient variety alone but the logarithmic pair consisting
of the quotient variety and the branch divisor with suitable rational coefficients.
An extreme case is the case where $G$ is a complex reflection group, i.e., it is generated by pseudoreflections.
For such a group, the quotient variety $X:=\bA^n/G$ is again isomorphic to $\bA^n$ by a famous theorem of Chevalley, Shephard and Todd,
and the $\bQ$-divisor we consider is of the form $B=\sum_i\frac{m_i-1}{m_i}D_i$ on $X$
where $D=\sum_i D_i$ is the discriminant divisor on the quotient $\bA^n$ (decomposed into irreducible components $D_i$)
and $m_i$ is the order of the pointwise stabilizer subgroup of a hyperplane of $\bA^n$ mapped to $D_i$.
In this article, we consider the case of real reflection groups of rank $3$.
In this case, the boundary divisor is of the form $B=\frac{1}{2}D$ since any reflection is of order $2$.
As a derived McKay correspondence for real reflection groups of rank $3$, we prove the following.
\begin{theorem}\label{thm:equivalence}
For a real reflection group $G$ of rank $3$, there exists a maximal resolution $Y$ of the pair $(X, \frac{1}{2}D)$
such that
\begin{enumerate}
\item the morphism $Y \to X$ is a composite of blowups along smooth curves isomorphic to $\bA^1$,
\item if  $\Dtilde=\sum_i\Dtilde_i$ denotes the strict transform of $D$ to $Y$ then $\Dtilde$ is smooth,
\item each $\Dtilde_i$  is isomorphic to $\bA^2$ blown up along points finetely many times,
\item if $\scY$ denotes the Deligne-Mumford stack associated with the pair $(Y, \frac{1}{2}\Dtilde)$,
then there is an equivalence
\begin{equation}\label{eq:mckay_for_reflection}
D(\scY) \cong D^G(\bA^3).
\end{equation}
\end{enumerate}
\end{theorem}
Here, a maximal resolution means a maximal $\bQ$-factorial terminalization which is smooth as a variety.
While a maximal $\bQ$-factorial terminalization always exists by \cite[Corollary 1.4.3]{BCHM}, the existence of a maximal resolution
is a non-trivial question in dimension $\ge 3$.
In fact, the construction in \cite{kawamata_GL3} directly applied to this case  produces singular
maximal $\bQ$-factorial terminalizations (see Remark \ref{rem:singular}), and we have to modify the construction.

Theorem \ref{thm:equivalence} has an application to a problem on semiorthogonal decompositions of equivariant derived categories.
The explicit description of $Y$ yields a semiorthogonal decomposition of the derived category $D(Y)$ into the derived
categories of $X$ and the blowup centers by \cite{Orlov_SOD}.
Moreover, we have $D(\scY)=\langle D(\scD), D(Y) \rangle$ by the semiorthogonal decomposition \eqref{eq:sod} for the root stack $\scY \to Y$.
Combining these semoiorthogonal decompositions and descriptions of $\Dtilde$ with \eqref{eq:mckay_for_reflection}, we obtain the following theorem.
\begin{theorem}\label{thm:sod}
Let $G$ be a real reflection group of rank $3$
acting linearly on $\bA^3$,
and put $X:=\bA^3/G \cong \bA^3$.
Then the equivariant derived category $D^G(\bA^3)$ has a semiorthogonal decomposition of the form
\[
D^G(\bA^3)=\langle D(Z_1), \dots, D(Z_m), D(X) \rangle
\]
where $Z_i$ are isomorphic to $\bA^0$, $\bA^1$ or $\bA^2$.
Moreover, for $i=0, 1, 2$, the number of $Z_j$ isomorphic to $\bA^i$ coincides
with the number of conjugacy classes $[g]$ whose fixed point loci $(\bA^3)^g$are of dimension $i$.
\end{theorem}
This theorem proves the following conjecture of Polishchuk and Van den Bergh
in the case of the linear action of a real reflection group of rank $3$,
since the quotient of $(\bA^3)^g\cong \bA^i$ by the centralizer $C(g)$ is isomorphic to $\bA^i$
(see \cite[Proposition 2.2.6]{PVdB}).

\begin{conjecture}[\cite{PVdB}]\label{conj:pvdb}
Suppose a finite group $G$ acts effectively on a smooth quasiprojective variety $X$.
Assume for any $g \in G$ that the quotient $X^g/C(g)$ of the invariant subvariety $X^g$ by the centralizer $C(g)$ is smooth.
Then there exists a semiorthogonal decomposition
\[
D^G(X) \cong \langle \scC_{[g]} \rangle_{[g]\in \Conj G}
\]
indexed by the set $\Conj G$ of the conjugacy classes of $G$
such that $\scC_{[g]} \cong D(X^g/C(g))$.
\end{conjecture}
In the case of the linear action of $G \subset \GL(n, \bC)$ on $\bA^n$,
the smoothness assumption of $X^g/C(g)$ for $g=e$ implies
$G$ is generated by pseudoreflections, that is, $G$ is a complex reflection group.
Polishchuk and Van den Bergh proved that the conjecture
(together with the smoothness of the quotients) is true
for the Weyl groups of type $A_n, B_n, G_2$ and $F_4$
(and the smoothness assumption fails for Weyl groups of type $D_n$)
as well as the complex reflection groups $G(m, 1, n)$.
They found other groups for which the smoothness assumption of $X^g/C(g)$
holds \cite[Proposition 2.2.6]{PVdB} and which were left for future work: complex reflection groups of rank $2$,
real reflection groups of rank $3$, and the group of type $H_4$.
Theorem \ref{thm:sod} establishes the case of real reflection groups of rank $3$.

We notice that for complex reflection groups of rank $2$, the conjecture follows from Kawamata's work
(see Corollary \ref{cor:kawamata}).
Recently,  a more explicit description of the decomposition for some reflection groups of rank $2$
was given in \cite{rank2} by Bhaduri, Davidov, Faber, Honigs, McDonald, Overton-Walker and Spence.
For dihedral groups, there are related works by Potter \cite{potter} and Capellan \cite{capellan}.
There is also a version of the McKay correspondence for complex reflection groups by Buchweitz, Faber, and Ingalls \cite{BFI} in terms of maximal Cohen Macaulay modules over the discriminant divisor.
It would be interesting to investigate its relation with constructions in Theorem \ref{thm:equivalence}.

We now explain the stragegy for the proof of Theorem \ref{thm:equivalence}.
We may suppose $G$ is not conjugate to a subgroup of $\GL(2, \bC)$.
Put $H:=G \cap \SL(3, \bC)$ and $X_H:=\bA^3/H$.
Together with the maximal resolution $f:Y \to X$, we construct a
($G/H$)-equivariant crepant resolution
\begin{equation}\label{eq:crepant}
f_H: Y_H \to X_H
\end{equation}
such that 
$Y$ and $\scY$ are isomorphic to the quotient variety $Y_H/(G/H)$ and the quotient stack $[Y_H/(G/H)]$ respectively.
\begin{equation}\label{eq:zushiki}
\begin{tikzpicture}[auto]
\node (m) at (0,0) {$\bA^3$}; \node[right=2 of m.center] (x0)  {$X_H$}; \node[right=2 of x0.center] (x) {$X$};
\node[above=1 of x0.center] (yh) {$Y_H$}; \node[above=1 of x.center] (x1) {$Y$};
\draw[->] (m) to node {$\pi_H$} (x0);
\draw[->] (yh) to node[swap] {$f_H$} (x0);
\draw[->] (x0) to node {$\pi_{G/H}$} (x);
\draw[->] (yh) to node {$\pi'_{G/H}$} (x1);
\draw[->] (x1) to node (x) {$f$} (x);
\end{tikzpicture}
\end{equation}
To construct $f:Y \to X$ and $f_H:Y_H \to X_H$,  we use one of
two different methods depending on the classification of groups in \S \ref{subsec:classify}.
For the groups other than the octahedral and the icosahedral groups,
we first construct $Y_H$ as an iterated Hilbert scheme
\[Y_H:=(H/K)\hilb(K\hilb(\bA^3))\]
for a suitable subgroup $K\subset H$ which is normal in $G$.
In this case, $Y_H$ is smooth by \cite{BKR}
and there is an equivalence $D^G(\bA^3) \cong D^{G/H}(Y_H)$
as in \cite[Theorem 4.1]{Ishii-Ueda}.
Verification of the smoothness of $Y:=Y_H/(G/H)$ and the description of $Y$ as
an iterated blowup of $X$ is done explicitly in each case.
This method is not always applicable since, for example, $H$ is a simple group in the icosahedral case.
For the octahedral and the icosahedral groups, we first obtain $Y$ as a successive blowup of $X$ along smooth curves
in the singular locus of the discriminant divisor $D$ and then construct $Y_H$ as a double covering of $Y$ branched along the strict transform $\Dtilde$.
This method was used to construct a crepant resolution of $\bA^3/H$ by
Bertin and Markushevich  \cite{MR1273078} in the tetrahedral and the octahedral cases, and by Roan \cite{MR1284568} in the icosahedral case.
In the octahedral and the icosahedral cases, we need to verify the equivalence $D^G(\bA^3) \cong D^{G/H}(Y_H)$ without using iterated Hilbert schemes.
For these groups, noting that $G=H\times\{\pm1\}$ holds, we apply  Proposition \ref{prop:yamagishi} below, which follows from Yamagishi's theorem \cite{MR4871720} stating that any projecrive crepant resolution of $\bA^3/H$ can be obtained as the moduli space of $G$-constellations for a suitable stability parameter.

The organization of the paper is as follows.
In \S \ref{subsec:max}, we recall the notion of a maximal $\bQ$-factorial terminalization and then discuss its construction from a crepant resolution $Y_H$ in the case of real reflection groups in Lemma \ref{lem:max}.
In \S \ref{subsec:ghilb}, we recall the properties of $G\hilb$ and the moduli space of $G$-constellations.
In particular, we prove Proposition \ref{prop:yamagishi} which is used in \S \ref{subsec:octahedral}
and \S \ref{subsec:icosahedral}.
In  \S \ref{subsec:sod}, we recall the semiorthogonal decomposition of derived categories in the cases of blowups and root stacks.
In \S \ref{sec:2-dim}, we recall Kawamata's works and explain Corollary \ref{cor:kawamata},
which is Conjecture \ref{conj:pvdb} in the case of complex reflection groups of rank $2$.
In \S \ref{sec:main}, we prove Theorems \ref{thm:equivalence} and \ref{thm:sod}.
We first recall the classification of real reflection groups of rank $3$ in \S \ref{subsec:classify}.
There are five types of real reflection groups of rank $3$ that are not conjugate to subgroups
of $\GL(2, \bC)$.
We prove one case in one subsection.

\emph{Acknowledgement}:
The first author thanks Kazushi Ueda for discussions on Conjecture \ref{conj:pvdb}.
He is also grateful to Yujiro Kawamata for his explanations about maximal $\bQ$-factorial terminalizations.
A.~I. is supported by Grant-in-Aid for Scientific Research (No.19K03444).
This work was supported by the Research Institute for Mathematical Sciences,
an International Joint Usage/Research Center located in Kyoto University.
\section{Preliminaries}\label{sec:pre}
We prove some results needed to prove the main theorems.
\subsection{Maximal $\bQ$-factorial terminalizations}\label{subsec:max}
For a KLT pair $(X, B)$, a projective birational morphism $f: Y \to X$ is called a {\it maximal $\bQ$-factorial terminailzation} \cite{kawamata_GL3} of $(X, B)$
if $Y$ has only $\bQ$-factorial and terminal singularities, and the exceptional divisors of $f$ are the prime divisors over $X$ with non-positive discrepancies. 
In \cite{BCHM}, it is called a terminal model, and its existence is proved in \cite[Corollary 1.4.3]{BCHM}.
We say $Y$ is a {\it maximal resolution} if $Y$ is smooth.
\begin{lemma}\label{lem:max}
For a real reflection group $G$ of rank $3$,
put $X:=\bA^3/G$, $H:=G\cap \SL(3, \bC)$, $X_H:=\bA^3/H$,
and let $D\subset X$ be the discriminant divisor.
Suppose
\[
f_H: Y_H \to X_H
\]
is a $(G/H)$-equivariant projective crepant resolution and put $Y:=Y_H/(G/H)$.
\begin{enumerate}
\item
The fixed point locus $Y_H^{(G/H)}$ is the disjoint union of a finite number of smooth surfaces and points.
\item
$Y$ is a maximal $\bQ$-factorial terminalization of $(X, \frac{1}{2}D)$.
\item
$Y$ is a maximal resolution of $(X, \frac{1}{2}D)$ if and only if $Y_H^{(G/H)}$ contains no isolated points.
\end{enumerate}
\end{lemma}
\begin{proof}
(1) Let $\sigma \in G/H$ be the generator.
We first prove that any connected component of the fixed point locus $(Y_H)^\sigma$ has odd codimension.
Consider the $H$-invariant $3$-form $dx\wedge dy \wedge dz$ on $\bA^3$,
which uniquely descends to a nowhere vanishing $3$-form $\omega_0$ on the smooth locus of $X_H$. Since $f_H$ is a crepant resolution,
$\omega_0$ extends to a nowhere vanishing $3$-form $\omega$ on $Y_H$.
On $\bA^3$, observe that a reflection pulls $dx\wedge dy \wedge dz$ back to
$-(dx\wedge dy \wedge dz)$,
which implies that
$\sigma^*\omega = -\omega$.
Therefore, if $P$ is a point of $(Y_H)^\sigma$, then the action of the involultion $\sigma$
on the Zariski tangent space $T_{Y_H, P}$  has $-1$ as its eigenvalue with odd multiplicity.
Thus we conclude that $(Y_H)^\sigma$ is the disjoint union of smooth surfaces and isolated points.

(2)
By (1), the quotient $Y_H/(G/H)$ has at worst quotient singularities of type $\frac{1}{2}(1,1,1)$
corresponding to isolated fixed points of $\sigma$ and thus $Y$ has only $\bQ$-factorial and terminal singularities.

We show that no two-dimensional component of $(Y_H)^\sigma$
is contained in the exceptional locus of $f_H$.
Recall that exceptional divisors of $f_H$ are in one-to-one correspondence
with junior conjugacy classes of $H$ \cite{IR}.
Since any element of $H$ is a rotation, its fixed point locus is not isolated.
This implies that the image of an exceptional divisor of $f_H$ has to be one-dimensional.
Assume that a two-dimensional component $S$ of $(Y_H)^\sigma$ is contained in the exceptional locus of $f_H$.
Then $C:=f_H(S)$ has to be an irreducible component of $\Sing X_H$
and hence $C$ is the image of a rotation axis $L \subset \bA^3$.
If $H_L \subset H$ denotes the pointwise stabilizer subgroup of $L$,
then $\bA^3/H_L \to \bA^3/H$ is {\' e}tale over a neibouhood of $C \setminus \{0\}$.
Since $\bA^3$ decomposes as $\bA^3=L \oplus L^{\perp}$ for some two-dimensional
subspace $L^{\perp}$ as a representation of the cyclic group $H_L$, 
$\bA^3/H_L \cong L \times (L^{\perp}/H_L)$ is the product of $L$ with
a two-dimensional singularity of type $A$.
Then, the induced action of $\sigma$ on the crepant resolution of $\bA^3/H_L$
comes from a reflection in a dihedral group in $\GL(2, \bC)$.
However, a reflection in a dihedral group does not fix every point on an exceptional curve
of the crepant resolution (see \cite[Theorem 6.5.4]{potter} or Example \ref{ex:dihedral} below)
and hence the same statement is true for the action of $\sigma$ on the fiber $f_H^{-1}(c)$ for $c \in C \setminus \{0\}$.
This contradicts our assumption and $S$ is not contained in the exceptional locus of $f_H$.

Now consider the pair $(Y, \frac{1}{2}\Dtilde)$ where $\Dtilde \subset Y$ is the strict transform of $D$.
It follows from what we have seen that $\Dtilde$ is smooth, is contained in the smooth locus of $Y$, and is the branch divisor of the quotient morphism $Y_H \to Y$.
Then the crepantness of $f_H$ implies the equality $K_Y + \frac{1}{2}\Dtilde =f^*(K_X + \frac{1}{2}D)$,
which means that the discrepancies of exceptional divisors of $f$ with respect to the pair $(X, \frac{1}{2}D)$ are all zero.
Since the pair $(Y, \frac{1}{2}\Dtilde)$ is also terminal, we conclude that $Y$ is a maximal $\bQ$-factorial terminalization.

(3) As observed in (1), the singular points of $Y$ are exactly the images of the isolated fixed points.
This implies (3).
\end{proof}
\begin{remark}\label{rem:singular}
If we put $Y_H=H\hilb(\bA^3)$ as in \cite{kawamata_GL3}, we can observe that the quotient $Y=Y_H/(G/H)$ has singularities whenever $G$ is not conjugate to a subgroup of $\GL(2, \bC)$ (see Remarks \ref{rem:Z_2^3}, \ref{rem:dihedral}, \ref{rem:tetra}, \ref{rem:octa} and \ref{rem:icosa}).
Thus, to obtain a maximal resolution, we have to construct a different crepant resolution.
\end{remark}
\subsection{$G$-Hilbert scheme and its variants}\label{subsec:ghilb}
For a finite subgroup $G \subset \GL(n, \bC)$, the $G$-Hilbert scheme \cite{Nakamura} \cite{Ito-Nakamura} is defined to
be the moduli scheme of {\it $G$-clusters}, where a $G$-cluster is a $G$-invariant finite subscheme $Z \subset \bC^n$ such that $H^0(\scO_Z)$ is isomorphic to the regular representation of $G$.
When $G \subset \SL(n, \bC)$ for $n \le 3$, by the work of Bridgeland, King and Reid \cite{BKR},
$Y:=G\hilb(\bC^n) \to X:=\bC^n/G$ is a crepant resolution and the Fourier-Mukai functor induces an equivalence
\[
D(Y) \cong D^G(\bA^n)
\]
between the derived category $D(Y)$ of coherent sheaves on $Y$ and the derived category $D^G(\bA^n)$
of $G$-equivariant coherent sheaves on $\bA^n$.

As a generalization of a $G$-cluster,  a {\it $G$-constellation} \cite{Craw-Ishii} is defined to be a
$G$-equivariant coherent sheaf $F$ on $\bC^n$ with finite support such that $H^0(F)$ is isomorphic to the regular representation of $G$.
For a parameter $\theta \in \Hom_{\bZ}(R(G), \bQ)$ where $R(G)$ is the representation ring of $G$ (regarded as a module), the $\theta$-stability of $G$-constellations is defined and there is a fine moduli scheme $\scM_\theta$ of $\theta$-stable $G$-constellations.
A stability parameter $\theta$ is said to be {\it generic} if the $\theta$-semistability implies the $\theta$-stability.
When $G \subset \SL(n, \bC)$ with $n \le 3$ and when $\theta$ is generic, arguments in \cite{BKR} also show that $Y=\scM_\theta$ is a crepant resolution of $\bA^n/G$ and there is an equivalence $D(Y) \cong D^G(\bA^n)$ \cite{Craw-Ishii}.

An example of the moduli space $\scM_\theta$ of $G$-constellations is constructed as an iterated Hilbert scheme \cite{IIN}.
If $N$ is a normal subgroup of $G$, then one can construct the iterated Hilbert scheme
$(G/N)\hilb(N\hilb(\bC^n))$.
 Assume $n \le 3$ and $N \subset \SL(n, \bC)$.
It is proved in \cite{IIN} that this iterated Hilbert scheme is canonically isomorphic to $\scM_\theta$ for a suitable choice of a generic stability parameter $\theta$.
An advantage of the iterated Hilbert scheme
is that it induces the following equivalence \cite[Theorem 4.1]{Ishii-Ueda}
\begin{equation}\label{eq:equivFM}
D^G(\bA^n) \cong D^{G/N}(N\hilb(\bA^n)).
\end{equation}
We use the iterated Hilbert scheme as the crepant resolution $Y_H$ in \eqref{eq:crepant} for several cases
in the classification of groups, but we use a different method to construct such a crepant resolution in the other cases (for example, there is no non-trivial iterated Hilbert scheme for a simple group).
In such cases, the next proposition, which follows from Yamagishi's result \cite{MR4871720}, is used:
\begin{proposition}\label{prop:yamagishi}
Let $G \subset \GL(3, \bC)$ be a finite subgroup of the form $G=A \times H$ where $A$ and $H$ are subgroups of $\GL(3, \bC)$ such that $H \subset \SL(3, \bC)$.
Let $Y$ be a projective crepant resolution of $\bA^3/H$.
Then the birational action of $A$ on $Y$, induced from the action on $\bA^3/H$, is a regular action, 
and there is an equivalence of triangulated categories:
\[
D^G(\bA^3) \cong D^A(Y).
\]
\end{proposition}
\begin{proof}
The condition $G=A \times H$ implies that $A \cap H=\{1\}$ and that an elements of $A$ and $H$ are commutative.
For an element $a \in A$ and an $H$-constellation $E$ on $\bA^3$,  $a^*E$ is also an $H$-constellation by the commutativity of elements.
Moreover, for any stability parameter $\theta$, $E$ is $\theta$-stable if and only if so is $a^*E$.
In particular, for any stability parameter $\theta$, $A$ acts on the moduli space $\scM_\theta$.
Now, as proved by Yamagishi \cite{MR4871720}, $Y$ is isomorphic to $\scM_\theta$ for some $\theta$.
Thus there is a regular action of $A$ on $Y$.

For the derived equivalence, consider the universal family $\scU$ of $\theta$-stable $H$-constellations.
Here $\scU$ is an $H$-equivariant coherent sheaf on $Y \times \bA^3$ flat over $Y$
such that for any $y \in Y$ the fiber over $y$ is the $H$-constellation corresponding to $y$.
Its direct image to $Y$ decomposes as
\[
p_{Y*}\scU =\bigoplus_{\rho \in \Irrep H}\scU_\rho \otimes_{\bC} \rho
\]
for some locally free sheaves $\scU_\rho$ of rank $\dim \rho$ on $Y$,
where $\Irrep H$ denotes the set of the irreducible representations of $H$
and  where $p_Y$ is the projection $Y \times \bA^3 \to Y$.
The locally free sheaves $\scU_\rho$ are called the tautological bundles.
Notice that the universal family is determined up to tensor products with line bundles on $Y$.
In fact, the action of $a \in A$ on $Y$ is determined by
\[
(a \times \id_{\bA^3})^*\scU \cong (\id_{Y} \times a^{-1})^*\scU \otimes p_Y^*L_a
\]
for some line bundle $L_a$ on $Y$.
We take the universal family $\scU$ satisfying the normalization condition $\scU_{\rho_0} \cong \scO_Y $
for the trivial representation $\rho_0$.
Fix an isomorphism $u: \scU_{\rho_0} \overset{\sim}{\to} \scO_Y$.
Then, for any $a \in A$, $L_a$ is canonically isomorphic to $\scO_Y$ and we see
\[
(a \times a)^* \scU \cong (\id_Y \times a)^*(a \times \id_{\bA^3})^*\scU \cong (\id_Y \times a)^*(\id_Y \times a^{-1})^*\scU \cong \scU.
\]
This defines an isomorphism of $H$-equivariant sheaves
\[
\lambda_a: \scU\overset{\sim}{\to}  (a\times a)^*\scU
\]
such that its $\rho_0$ part $(\lambda_a)_{\rho_0}$ coincides with $(a^*u)^{-1}\circ u$.
Then it is easy to see that the collection $(\lambda_a)_{a \in A}$ defines an $A$-equivariant action on the $H$-equivariant sheaf $\scU$.
In other words, $\scU$ has a structure of an $A\times H$-equivariant sheaf on $Y \times \bA^3$.
Now the equivalence in the statement follows by the same argument as in \cite[Theorem 4.1]{Ishii-Ueda}.
\end{proof}
\subsection{Semiorthogonal decompositions}\label{subsec:sod}
We first recall the definition.
For a triangulated category $\scD$, a semiorthogonal decomposition of $\scD$ consists of full triangulated subcategories
$\scC_1, \dots, \scC_m$ such that
\begin{itemize}
\item For any $C_i \in \scC_i$ and $C_j \in \scC_j$ with $i>j$, $\Hom(C_i, C_j)=0$.
\item For any object $F$ in $\scD$, there are objects $C_i \in \scC_i$ and $F_2, \dots, F_m \in \scD$ with distinguished triangles
$F_{i+1} \to F_i \to C_i \to F_{i+1}[1]$ for $i=1,\dots m$ where $F_1:=F$ and $F_{m+1}:=0$.
\end{itemize}
We write
\[
\scD=\langle \scC_1, \dots, \scC_m \rangle
\]
for a semiorthogonal decomposition.
By abuse of notation, we write $\scD \cong \langle \scC_1, \dots, \scC_m \rangle$
even when the embedding functors of $\scC_i$ into $\scD$ are not specified.
\begin{remark}
One can mutate a semiorthogonal decomposition if the subcategories are admissible \cite{BK}.
In fact, all the subcategories appearing in this article are admissible.
If the new subcategories created by mutations are again admissible, one can repeat mutations
and permute the components in an arbitrary order. 
However, the authors did not check this condition and
we do not freely permute components of semiorthogonal decompositions in this article.
\end{remark}

For a smooth subvariety $Z$ of a smooth quasiprojective variety $X$ with codimension $c$, let $\Xtilde$ be the blowup of $X$ along $Z$.
Then \cite[Theorem 4.3]{Orlov_SOD} states that there is a semiorthogonal decomposition
\begin{equation}\label{eq:blowup}
D(\Xtilde) \cong \langle D(Z), \dots, D(Z), D(X) \rangle
\end{equation}
in which $D(Z)$ appears $c-1$ times.

For a smooth divisor $\scD$ on a smooth Deligne-Mumford stack $\scX$,
there is a notion of the $r$-th root stack \cite{Cadman} \cite{AGV} $\scY=\sqrt[r]{(\scO_{\scX}(\scD), 1)}\to \scX$
which is an isomorphism on $\scX\setminus \scD$ and which is ramified along $\scD$ with ramification index $r$.
Then by \cite[Theorem 1.6]{Ishii-Ueda} or \cite[Theorem 4.7]{BLS}, there is a semiorthogonal decomposition
\begin{equation}\label{eq:sod}
D(\scY)\cong \langle D(\scD), \dots, D(\scD), D(\scX) \rangle
\end{equation}
in which $D(\scD)$ appears $r-1$ times.
If $r=2$ as we consider in this article, then the semiorthogonal decomposition is $D(\scY)\cong \langle D(\scD), D(\scX) \rangle$.
\section{Semiorthogonal decompositions for complex reflection groups of rank $2$}\label{sec:2-dim}

Let $G \subset \GL(2, \bC)$ be a finite subgroup, let $\pi:\bA^2 \to X:=\bA^2/G$ be the quotient morphism, and let $Y \to X$ be the minimal resolution.
Define a $\bQ$-divisor $B$ on $X$ such that $\pi^*(K_X+B)=K_{\bA^2}$.
Then $B$ is of the form
\[
B=\sum_{i=1}^r \frac{m_i-1}{m_i}D_i,
\]
where $D_i=\pi(L_i)$ for some line $L_i \subset \bA^2$ with $\{g \in G \mid g |_{L_i}=\id\}\cong \bZ/m_i\bZ$.
Under this setting, Kawamata proved the following:
\begin{theorem}[\cite{kawamata_toric_III} Theorem1.4]\label{thm:kawamata}
There is a semiorthogonal decomposition
\begin{equation}\label{eq:kawamata}
D^G(\bA^2) \cong \langle D(Z_1), \dots, D(Z_m), D(Y) \rangle
\end{equation}
where $Z_i$ are isomorphic to either $\bA^0$ or $\bA^1$ and
the number of times $D(\bA^1)$ appears is $\sum_{i=1}^r (m_i-1)$.
\end{theorem}
Here the number of copies of $D(\bA^1)$ in \eqref{eq:kawamata} is not stated in \cite[Theorem 1.4]{kawamata_toric_III} but
it follows from the proof;
it is stated in \cite[p.205]{kawamata_toric_III} that the positive dimensional subvarieties appear in the process to decrease the coefficients of $B$,
and the number is counted in the proof of \cite[Theorem 1.1(a)]{kawamata_toric_III}.
Notice that the number $\sum_{i=1}^r (m_i-1)$ coincides with the number of conjugacy classes of pseudoreflections in $G$.
\begin{corollary}\label{cor:kawamata}
Conjecture \ref{conj:pvdb} is true for complex reflection groups of rank $2$.
\end{corollary}
\begin{proof}
We have an isomorphism of the Grothendieck groups
\[
K^G(\bA^2) \cong R(G)
\]
 by Bass-Haboush \cite{Bass-Haboush}, where $R(G)$ is the representation ring of $G$.
Since $K(\bA^i) \cong \bZ$ for $i \ge 0$ and since $Y=X \cong \bA^n$ for a complex reflection group, this implies that the number of the components in the semiorthogonal decomposition \eqref{eq:kawamata}
 coincides with the number of conjugacy classes of $G$.
 Therefore, the number of copies of $\bA^0$ in \eqref{eq:kawamata} coincides with the number of conjugacy classes
 whose fixed point set is $\{0\}$.
 This proves Polischchuk-Van den Bergh conjecture in dimension two.
 \end{proof}
 \begin{remark}
 This proves the coincidence of numbers but it is not clear how to index each component $D(\bA^0)$ by a conjugacy class.
 \end{remark}
 \begin{example}\label{ex:Z_2^2}
 The smallest reflection group in $\GL(2, \bC)$ which is not conjugate to a subgroup
 of $\GL(1, \bC)$ is the following group:
 \[
 G=\{ \diag(\pm 1, \pm 1) \}. 
 \]
 This is conjugate to the group we obtain putting $n=2$ in Example \ref{ex:dihedral} below,
 but we describe this case to use it in \S \ref{subsec:dihedral}.
 The branch divisor $D$ of the quotient morphism  $\bA^2 \to X$ is expressed as $D=D_1+D_2$, where $D_i$ are the coordinate axes of $X=\bA^2/G \cong \bA^2$.
 Put $H:=G \cap \SL(2, \bC) \cong \bZ/2\bZ$, $X_H:=\bA^2/H$ and $Y_H:=H\hilb(\bA^2)$.
 Then $X_H$ is the $A_1$ singularity, $Y_H$ is its minimal resolution,
 and $Y:=Y_H/(G/H)$ is the maximal resolution of $(X, \frac{1}{2}D)$.
 Since $Y$ is the blowup of $X$ along the origin,
 we have a semiorthogonal decomposition
 \[
 D(Y) \cong \langle D(\bA^0), D(X) \rangle.
 \]
 The strict transform $\Dtilde=\Dtilde_1+\Dtilde_2$ of $D$ in $Y$ is smooth
 and the quotient stack $[Y/(G/H)]$ is isomorphic to the $2$nd root stack
 of $Y$ along $\Dtilde$.
This induces a semiorthogonal decomposition
\[
D^{G/H}(Y_H) \cong \langle D(\bA^1), D(\bA^1), D(Y)\rangle.
\] 
Finally, \eqref{eq:equivFM} implies
\[
D^G(\bA^2) \cong D^{G/H}(Y_H).
\]
Summarizing the above decompositions, we obtain
\[
D^G(\bA^2) \cong \langle D(\bA^1), D(\bA^1), D(\bA^0), D(X) \rangle.
\]
One can see that this decomposition can be actually indexed by conjugacy classes of $G$.
 \end{example}
 \begin{example}\label{ex:dihedral}
For $n \ge 3$, we consider the dihedral group
 \[
D_{2n}:=\left\langle\begin{pmatrix} \zeta_n & 0 \\ 0 & \zeta_n^{-1} \end{pmatrix}, \begin{pmatrix} 0 & 1 \\ 1 &0 \end{pmatrix} \right\rangle \subset \GL(2, \bC)
\]
where $\zeta_n=\exp{\frac{2\pi\sqrt{-1}}{n}}$.
This case was studied by Potter \cite{potter} and further by Capellan \cite{capellan}.
Put $H:=G\cap \SL(2, \bC)=\left\langle\begin{pmatrix} \zeta_n & 0 \\ 0 & \zeta_n^{-1} \end{pmatrix}\right\rangle$
and $Y_H:=H\hilb(\bA^2)$.
Then $Y_H$ is the minimal resolution of the quotient singularity $X_H:=\bA^2/H$ of type $A_{n-1}$ and there are $n-1$ exceptional curves $E_1, \dots, E_{n-1}$ such that
$E_{i-1} \cap E_i$ consists of a point. 
Connected components of the fixed point locus of the action of $G/H$ on $Y_H$
are of codimension one as in Lemma \ref{lem:max} (1).
Moreover, the action of $G/H$ exchanges the strict transforms in $Y_H$ of the images of the two coordinate axes of $\bA^2$, one intersecting $E_1$
and the other intersecting $E_{n-1}$.
Then, it also exchanges  $E_i$ with $E_{n-i}$,
and no exceptional curve $E_i$ is contained in the fixed point set $(Y_H)^{G/H}$.
Therefore, $(Y_H)^{G/H}$ coincides with the strict transform $\Dtilde \subset Y$ of the discriminant divisor $D \subset X$.
This implies that $Y:=Y_H/(G/H)$ is smooth and is obtained by blowing up $X=\bA^2/G\cong\bA^2$ along $\lfloor \frac{n}{2} \rfloor$ points. 
Therefore, we obtain a semiorthogonal decomposition
\[
D(Y) \cong \langle D(\bA^0), \dots, D(\bA^0), D(X) \rangle
\]
in which $D(\bA^0)$ appears $\lfloor \frac{n}{2} \rfloor$ times.
The quotient stack $[Y_H/(G/H)]$ is the $2$nd root stack of $Y$ along the smooth divisor $\Dtilde\subset Y$, and
 $\Dtilde$ has one ($n$:odd) or two($n$:even) connected components isomorphic to $\bA^1$.
 This leads to the semiorthogonal decomposition
 \[
 D([Y_H/(G/H)]) \cong \begin{cases}
 							\langle D(\bA^1), D(Y) \rangle & \text{($n$:odd)}\\
							\langle D(\bA^1), D(\bA^1), D(Y) \rangle & \text{($n$:even)}.
\end{cases}
 \]
 Finally, \eqref{eq:equivFM} implies an equivalence $D^G(\bA^2) \cong D([Y_H/(G/H)])$.
Summarizing these decompositions, we obtain
\[
D^G(\bA^2) \cong \begin{cases}
 							\langle D(\bA^1), D(\bA^0), \dots, D(\bA^0), D(Y) \rangle & \text{($n$:odd)}\\
							\langle D(\bA^1), D(\bA^1), D(\bA^0), \dots, D(\bA^0), D(Y) \rangle & \text{($n$:even)}
						\end{cases}
\]
where $D(\bA^0)$ appears $\lfloor \frac{n}{2} \rfloor$ times.
It is easy to see that the number of copies of $D(\bA^i)$ in this decomposition
coincides with the number of conjugacy classes $[g]$ with $(\bA^2)^g \cong \bA^i$.
\end{example}
\section{Real reflection groups of rank $3$}\label{sec:main}
\subsection{Classification of real reflection groups of rank $3$}\label{subsec:classify}
The classification of real reflection groups up to conjugacy in $\GL(3)$ (see \cite{Coxeter} or \cite[19.4]{BB} for example) is as follows:
\begin{enumerate}
\item The diagonal subgroups $\bZ/2\bZ$, $(\bZ/2\bZ)^2$, $(\bZ/2\bZ)^3$.
\item The dihedral group $D_{2n}$ of order $2n$ ($n \ge 3$).
\item $D_{2n} \times \bZ/2\bZ$ ($n \ge 3)$.
\item The symmetry group of a tetrahedron (we call it the tetrahedral group).
\item The symmetry group of an octahedron (we call it the octahedral group).
\item The symmetry group of an icosahedron (we call it the icosahedral group).
\end{enumerate}
The groups $\bZ/2\bZ$, $(\bZ/2\bZ)^2$, and (2) are subgroups of $\GL(2)$, and maximal resolutions and semiorthogonal decompositions
can be obtained by applying the construction in the rank two cases, where we replace each variety $V$ with $V \times \bA^1$.
Thus we consider the remaining five cases one by one.
For the groups $(\bZ/2\bZ)^3$,  (3), (4) we use iterated Hilbert schemes
and for the groups  (5), (6) we use Proposition \ref{prop:yamagishi}.

\begin{remark}
The groups  (1) can be regarded as special cases of (2) or (3).
However, since these cases are easier to describe by the toric nature of the associated varieties,
we treat the case of $(\bZ/2\bZ)^3$ separately.
\end{remark}

\begin{remark}
The tetrahedral group (5) and the octahedral group (6) are Weyl groups of type $A$ and $B$ respectively.
Thus, Theorem \ref{thm:sod} for these groups follows from  \cite[Theorem C]{PVdB}.
However, it doesn't imply Theorem \ref{thm:equivalence} for these groups.
\end{remark}
\subsection{The case of $(\bZ/2\bZ)^3$}\label{subsec:abelian}
Suppose $G$ is the diagonal subgroup
\[G=\{\diag(\pm 1, \pm 1, \pm 1)\}.\]
In this case,
\[
H=G \cap \SL(3, \bC)\cong (\bZ/2\bZ)^2
\]
and $G$ decomposes as $G=H \times \langle -E \rangle$.
We further consider the subgroup
\[
K:=\langle \diag (1,-1,-1) \rangle \subset H
\]
and put
\begin{equation}\label{eq:xyhk}
\begin{aligned}
X_K&:=\bA^3/K,& Y_K&:=K\hilb(\bA^3) \\
X_H&:=\bA^3/H &Y_H&:=(H/K)\hilb(Y_K) \\
X&:=\bA^3/G\cong \bA^3 & Y&:=Y_H/(G/H).
\end{aligned}
\end{equation}
to obtain the following diagram.
\begin{equation}\label{eq:zushiki_K}
\begin{tikzpicture}[auto]
\node (m) at (0,0) {$\bA^3$}; \node[right=2 of m.center] (xk)  {$X_K$}; \node[right=2 of xk.center] (xh) {$X_H$};
\node[right=3 of xh.center] (x) {$X$};
\node[above=1 of x0.center] (yk) {$Y_K$}; \node[above=1 of xh.center] (w1) {$Y_K/(H/K)$};\node[above=1 of x.center] (w2) {$Y_K/(G/K)$};
\node[above=1 of w1.center] (yh) {$Y_H$};
\node[above=1 of w2.center] (y) {$Y$};
\draw[->] (m) to node {$\pi_K$} (xk);
\draw[->] (xk) to node {$\pi_{H/K}$} (xh);
\draw[->] (xh) to node {$\pi_{G/H}$} (x);
\draw[->] (yk) to node {$\pi'_{H/K}$} (w1);
\draw[->] (w1) to node {$\pi'_{G/H}$} (w2);
\draw[->] (yh) to (w1);
\draw[->] (y) to node[swap] {$f_2$}(w2);
\draw[->] (yk) to node[swap] {$f_K$} (xk);
\draw[->] (w1) to (xh);
\draw[->] (w2) to node[swap] {$f_1$} (x);
\draw[->] (yh) to node {$\pi''_{G/H}$}(y);
\end{tikzpicture}
\end{equation}
Then $Y_H$ is a $(G/H)$-equivariant crepant resolution of $X_H$,
$Y$ is a maximal $\bQ$-factorial terminalization of $(X, \frac{1}{2}D)$
for the discriminant divisor $D \subset X$ by Lemma \ref{lem:max},
and we have an equivalence
\begin{equation}\label{eq:GG/HY_H}
D^G(\bA^3) \cong D^{G/K}(Y_K) \cong D^{G/H}(Y_H).
\end{equation}
\begin{figure}[h]
 \begin{tikzpicture}
  \draw (0,0) node[left]{$e_1$} -- (0:2) node[right]{$e_2$}-- ++(120:2) node[above]{$e_3$} -- ++(240:2)  -- cycle;
   \draw (1,0) -- ++(60:1) -- ++(180:1);
     \draw (0,0)-- ++(30:sqrt 3);
 \end{tikzpicture}
 \caption{Toric picture for $Y_H$}\label{fig:toric_Y}
 \end{figure}
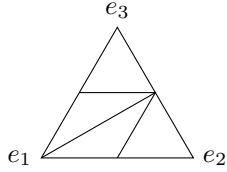
 The toric variety $Y_H$ corresponds to the fan determined by Figure \ref{fig:toric_Y}
 whose $N$-lattice is $\{(a, b, c) \in (\frac{1}{2}\bZ)^3 \mid a+ b+c \in \bZ\}$.
Its quotient $Y$ is determined by the same picture by replacing the $N$-lattice with $(\frac{1}{2}\bZ)^3$. Then one can see that $Y$ is also smooth and hence it is a maximal resolution of $(X, \frac{1}{2}D)$. 
 Let $D_i \subset X$ be the image of the coordinate hyperplane $\{(x_1, x_2, x_3) \in \bA^3 \mid x_i=0\} \subset \bA^3$
for $i=1, 2, 3$ and let  $\Dtilde_i \subset Y$ be the strict transform of $D_i$.
Then we see that $D=D_1+D_2+D_3$ with
$\Dtilde_2 \cong \Dtilde_3 \cong \bA^2$
and $\Dtilde_1 \cong \widetilde{\bA^2}$, the one-point blowup of $\bA^2$
This establishes Theorem \ref{thm:equivalence} (2)(3).
Moreover, one can see
\begin{itemize}
\item The double covering $\pi''_{G/H}:Y_H \to Y$ is ramified along $\Dtilde_1 \cup \Dtilde_2 \cup \Dtilde_3$ and hence
the stack $\scY$ in Theorem \ref{thm:equivalence} (4) is the quotient stack $[Y_H/(G/H)]$.
Thus the equivalence \eqref{eq:mckay_for_reflection} follows from \eqref{eq:GG/HY_H}.
Moreover, there is a semiorthogonal decomposition
\[
D^{G/H}(Y_H) \cong \langle D(\Dtilde_1), D(\Dtilde_2), D(\Dtilde_3), D(Y)\rangle.
\]
Since the derived category $D(\Dtilde_3)$ decomposes as $D(\Dtilde_3)\cong \langle D(\bA^0), D(\bA^2) \rangle$,
we obtain
\begin{equation}\label{eq:decompose_DG/H}
D^{G/H}(Y_H) \cong \langle D(\bA^0), D(\bA^2), D((\bA^2), D(\bA^2), D(Y)\rangle.
\end{equation}
\item The morphism $f:Y \to X$ is a composite of three blowups along affine lines ($f_1$ is a single blowup and $f_2$ is a composite of two blowups); this gives Theorem \ref{thm:sod} (1).
Thus $D(Y)$ decomposes as
\begin{equation}\label{eq:decompose_DY}
D(Y)\cong \langle D(\bA^1), D(\bA^1), D(\bA^1), D(X)\rangle.
\end{equation}
\end{itemize}
The equivalence \eqref{eq:GG/HY_H} and the decompositions \eqref{eq:decompose_DG/H}, \eqref{eq:decompose_DY} give the following semiorthogonal decomposition of $D^G(\bA^3)$:
\[
D^G(\bA^3) \cong \langle D(\bA^0), D(\bA^2), D(\bA^2), D(\bA^2), D(\bA^1), D(\bA^1), D(\bA^1), D(X) \rangle
\]
which is Theorem \ref{thm:sod} in this case.
\begin{remark}\label{rem:Z_2^3}
If we use $Y'_H:=H\hilb(\bA^3)$ instead of $Y_H$, then
the fixed point set $(Y'_H)^{G/H}$ consists of three divisors
corresponding to $e_1, e_2, e_3$ and the torus fixed point
corresponding to the central $3$-dimensional cone.
In particular, the quotient $Y':= Y'_H/(G/H)$ has a terminal quotient singularity
and we should consider the canonical stack $\frak Y'$ for the derived category.
Notice that the branch divisors of $Y'_H \to Y'$ are all isomorphic to $\bA^2$
and we obtain
\[
D^{G/H}(Y'_H) \cong \langle D(\bA^2), D(\bA^2), D(\bA^2), D(\frak Y') \rangle
\]
which is different from \eqref{eq:decompose_DG/H}.
Instead, the rational map $Y' \dasharrow Y$ is a toric flip and it provides a semiorthogonal decomposition
$D(\frak Y')\cong \langle D(\bA^0), D(Y)\rangle$ as proved by Kawamata \cite{kawamata_I}.
In this way, we can also obtain a semiorthogonal decomposition in which the order of the subcategories changes.
\begin{figure}[h]
 \begin{tikzpicture}
  \draw (0,0) node[left]{$e_1$} -- (0:2) node[right]{$e_2$}-- ++(120:2) node[above]{$e_3$} -- ++(240:2)  -- cycle;
   \draw (1,0) -- ++(60:1) -- ++(180:1) -- ++(300:1);
 \end{tikzpicture}
 \caption*{Toric picture for $H\hilb(\bA^3)$}
 \end{figure}
\end{remark}
\subsection{The case of $D_{2n}\times \bZ/2\bZ$}\label{subsec:dihedral}
Suppose
\[
G=D_{2n}\times \{\pm 1\} \subset \GL(2, \bC) \times \GL(1, \bC),
\]
where
$D_{2n}$ is the dihedral group of order $2n$ in Example \ref{ex:dihedral}.
We consider the subgroups
\begin{align*}
H&=G\cap \SL(3, \bC)=\left\langle
			\begin{pmatrix} \zeta_n & 0 & 0 \\ 0 & \zeta_n^{-1} &0 \\ 0&0&1\end{pmatrix},
			\begin{pmatrix} 0 & 1 & 0 \\ 1 & 0 & 0 \\ 0 & 0 & -1 \end{pmatrix}
			\right\rangle \cong D_{2n}
\\
K&:=\left\langle \begin{pmatrix} \zeta_n & 0 & 0 \\ 0 & \zeta_n^{-1} &0 \\ 0&0&1\end{pmatrix} \right\rangle
\subset H
\end{align*}
and define $X_K, Y_K, X_H, Y_H, X, Y$ as in \eqref{eq:xyhk}
to obtain the same diagram \eqref{eq:zushiki_K} where $Y$ becomes a maximal $\bQ$-factorial terminalization of $(X, \frac{1}{2}D)$.
We again have an equivalence
\begin{equation}\label{eq:dihedral_0}
D^G(\bA^3) \cong D^{G/K}(Y_K) \cong D^{G/H}(Y_H).
\end{equation}
By regarding $K$ as a subgroup of $\SL(2, \bC)$, we define $Y_K':=K\hilb(\bA^2)$.
Then $Y_K$ decomposes as
\[
Y_K=Y_K'\times \bA^1.
\]
We consider the action of $G/K \cong (\bZ/2\bZ)^2$ on $Y_K$.
The group $G/K$ is generated by the two elements
\[
\alpha:=\begin{pmatrix} 0 & 1 & 0 \\ 1 & 0 & 0 \\ 0 & 0 & 1\end{pmatrix}, \quad
\beta:=\diag(1, 1, -1)
\]
and the actions of these elements on $Y_K$ are products of actions on $Y'_K$ and $\bA^1$ which are described as follows:
\begin{itemize}
\item
The fixed point locus of $\alpha$ in $Y_K'$ is isomorphic to $\bA^1$ when $n$ is odd and isomorphic to the disjoint union of two copies of $\bA^1$ when $n$ is even.
The action of $\alpha$ on the second factor $\bA^1$ of $Y_K$ is trivial.
\item
$\beta$ acts trivially on $Y'_K$ and acts as a reflection on $\bA^1$.
\end{itemize}
In particular, the quotient of $Y_K$ by $H/K \cong \langle \alpha \beta \rangle\cong \bZ/2\bZ$ contains a family of $A_1$-singularities
(along one or two affine lines) and $Y_K/(G/K)$ is smooth.
The construction of $Y$ from the $(G/K)$-action on $Y_K$ is locally isomorphic to
the construction in Example \ref{ex:Z_2^2}, up to a trivial one-dimensional factor.
Therefore, the threefold $Y$ is smooth and obtained by blowing up $Y_K/(G/K)$ along one or two affine lines.
Thus we obtain
\begin{equation}\label{eq:dihedral_1}
D(Y)\cong
		\begin{cases} \langle D(\bA^1), D(Y_K/(G/K) \rangle & \text{($n$:odd)}\\
					 \langle D(\bA^1), D(\bA^1), D(Y_K/(G/K) \rangle & \text{($n$:even)}.
		\end{cases}
\end{equation}

We next describe the morphism $Y_K/(G/K) \to X=\bA^3/G$.
Since $G$ decomposes as $G=D_{2n}\times\{\pm 1\}$, 
we have $\bA^3/G=\bA^2/D_{2n} \times \bA^1/\langle \beta \rangle$
and $Y_K/(G/K) = (Y'_K/\langle \alpha \rangle) \times (\bA^1/\langle \beta \rangle)$.
Thus the morphism $Y_K/(G/K) \to X$ is induced from $Y'_K /\langle \alpha \rangle\to (\bA^2/D_{2n})$,
which is the composite of $\lfloor \frac{n}{2} \rfloor$ blowups at points.
Thus the morphism $Y_K/(G/K) \to X$  is the composite of $\lfloor \frac{n}{2} \rfloor$ blowups along affine lines.
Therefore we have a semiorthogonal decomposition
\begin{equation}\label{eq:dihedral_2}
D(Y_K/(G/K) )\cong \langle D(\bA^1), \dots, D(\bA^1), D(X) \rangle
\end{equation}
where $D(\bA^1)$ appears $\lfloor \frac{n}{2} \rfloor$ times.

So far, we have seen Theorem \ref{thm:equivalence} (1) holds in this case.
Finally, we have to describe the strict transform $\Dtilde$ of the discriminant divisor $D$.
We already know that $\Dtilde$ is smooth by Lemma \ref{lem:max} (1)
and that the equivalence \eqref{eq:mckay_for_reflection} holds by \eqref{eq:dihedral_0}.
Since $D$ is the union of the images of reflection hyperplanes in $\bA^3$,
$D$ has two irreducible components if $n$ is odd, and three irreducible components if $n$ is even.
Let $\Dtilde' \subset Y'_K/\langle \alpha \rangle$ be the branch divisor of $Y'_K \to Y'_K/\langle \alpha \rangle$.
Then the strict transform of $D$ to $Y_K/(G/K) \cong (Y'_K/\langle \alpha \rangle)\times \bA^1$ is
\[
\Dbar:=\Dtilde'\times \bA^1 + (Y'_K/\langle \alpha \rangle) \times \{0\},
\] 
and the morphism from $\Dtilde$ to $\Dbar$ induces isomorphisms between irreducible components  by Example \ref{ex:Z_2^2}.
Therefore, one component of $\Dtilde$ is isomorphic to $Y_K'/\langle \alpha \rangle$, which is $\lfloor \frac{n}{2} \rfloor$ points blowup of $\bA^2$ and
the other components are isomorphic to $\bA^2$. 
Hence we obtain
\begin{equation}\label{eq:dihedral_3}
D^{G/H}(Y_H) \cong
	\begin{cases}
		 \langle D(\bA^0), \dots, D(\bA^0), D(\bA^2), D(\bA^2), D(Y) \rangle & \text{($n$:odd)} \\
		 \langle D(\bA^0), \dots, D(\bA^0), D(\bA^2), D(\bA^2), D(\bA^2), D(Y)\rangle & \text{($n$:even)} 
	\end{cases}
\end{equation}
where $D(\bA^0)$ appears $\lfloor \frac{n}{2} \rfloor$ times.

Combining \eqref{eq:dihedral_0}, \eqref{eq:dihedral_1}, \eqref{eq:dihedral_2} and \eqref{eq:dihedral_3},
we obtain
\[
D^G(\bA^3)\cong
		 \langle D(\bA^0), \dots,  D(\bA^0), D(\bA^2), \dots, D(\bA^2), D(\bA^1), \dots, D(\bA^1), D(X)\rangle
 \]
 where the numbers of semiorthogonal components are described as follows:
 \begin{itemize}
 \item If $n$ is odd, then $D(\bA^2)$ appears twice, $D(\bA^1)$ appears $\frac{n+1}{2}$ times, and  $D(\bA^0)$ appears $\frac{n-1}{2}$ times.
 \item If $n$ is even ,then  $D(\bA^2)$ appears three times, $D(\bA^1)$ appears $\frac{n+4}{2}$ times, and  $D(\bA^0)$ appears $\frac{n}{2}$ times.
 \end{itemize}
 On the other hand, one can see that these numbers coincide with the numbers of conjugacy classes $[g]$ of $G$ such that $(\bA^3)^g$ are $\bA^2$, $\bA^1$ or $\bA^0$ respectively.
 Thus Theorem \ref{thm:sod} holds in this case.
 
 \begin{remark}\label{rem:dihedral}
 The $H$-Hilbert scheme $H\hilb(\bA^3)$ is investigated in \cite{GNS2} and
 the fiber of $H\hilb(\bA^3) \to X_H$ over the origin consists of $\frac{n+1}{2}$ curves
 when $n$ is odd and $\frac{n+4}{2}$ curves when $n$ is even (\cite[Table 3]{GNS2}). 
  The families of $G$-clusters parametrized by these curves are described in \cite[3.4, 3.5]{GNS2},
 and one can observe that there are exactly two fixed points by the action of the reflection $\beta$.
Then, the graph in \cite[Table 3]{GNS2} implies that the number of the fixed points on the fiber is $\frac{n+3}{2}$ when $n$ is odd
 and $\frac{n+6}{2}$ when $n$ is even.
 
 Let $\Dhat_i\subset H\hilb(\bA^3)/(G/H)$ denote the strict transform of an irreducible component $D_i$ of $D$.
Since $D_i$ is the image of some reflection hyperplane $L_i \subset \bA^3$ and since the fiber of $L_i \to D_i$ over the origin is a single point, the normalization of $D_i$ has a connected fiber over the origin.
This implies that the fiber of the birational morphism $\Dhat_i \to D_i$ over the origin is also connected.
Then it follows from the discreteness of the $\beta$-fixed points in the fiber of $H\hilb(\bA^3) \to X_H$
that the fiber of $\Dhat_i \to D_i$ over the origin has to be a single point.
Therefore, recalling that  $D$ has $2$ or $3$ irreducible components,
we see by Lemma \ref{lem:max} (1) that the fixed point locus $(H\hilb(\bA^3))^\beta$ contains $\lfloor\frac{n}{2}\rfloor$ isolated points.
In particular, $H\hilb(\bA^3)/(G/H)$ is singular for each $n$.
 \end{remark}
 \subsection{The case of the tetrahedral group}\label{subsec:tetrahedral}
 Let $G \subset \GL(3, \bC)$ be the subgroup consisting of elements that preserve the tetrahedron in $\bR^3$.
 We assume the four vertices of the tetrahedron are $(1, -1, -1)$, $(-1, 1, -1)$, $(-1, -1, 1)$, and $(1,1,1)$.
 Put
\begin{align*}
 K:&=\{\diag ( \pm 1, \pm 1, \pm 1)\} \cap \SL(3, \bC)\cong(\bZ/2\bZ)^2,  \\
 \alpha:&=\begin{pmatrix} 0 & 1 & 0 \\ 1 & 0 & 0 \\ 0 & 0 & 1 \end{pmatrix},\quad
 \gamma:=\begin{pmatrix} 0 & 0 & 1 \\ 1 & 0 & 0 \\ 0 & 1 & 0 \end{pmatrix}.
 \end{align*}
 Then we have $G=\langle K, \gamma, \alpha \rangle$ and
  \[
 H= G \cap \SL(3, \bC) = \langle K, \gamma \rangle \cong A_4.
 \]
 The invariant rings are described as
 \[ \bC[x_1, x_2, x_3]^G=\bC[g_1, g_2, g_3], \quad \bC[x, y, z]^H=\bC[g_1, g_2, g_3, g_4] \]
 where
 \begin{align*}
    g_1&=x_1^2+x_2^2+x_3^2,\\
    g_2&=x_1^2x_2^2+x_2^2x_3^2+x_3^2x_1^2,\\
    g_3&=x_1x_2x_3,\\
    g_4&=(x_1^2-x_2^2)(x_2^2-x_3^2)(x_3^2-x_1^2).
\end{align*}

 We define $X_K, Y_K, X_H, Y_H, X, Y$ as in \eqref{eq:xyhk}  and obtain an equivalence
\begin{equation}\label{eq:tetrahedral_0}
D^G(\bA^3) \cong D^{G/H}(Y_H).
\end{equation}
Since $K$ is an abelian group, $Y_K$ is a toric variety and it is determined by the fan corresponding to Figure \ref{fig:toric}
which have four three-dimensional cones. 
\begin{figure}[h]
  \begin{tikzpicture}
  \draw (0,0) node[left]{$e_1$} -- (0:2) node[right]{$e_2$}-- ++(120:2) node[above]{$e_3$} -- ++(240:2)  -- cycle;
   \draw (1,0) -- ++(60:1) -- ++(180:1) -- ++(300:1);
 \end{tikzpicture}
 \caption{Toric picture for $Y_K$}\label{fig:toric}
 \end{figure}
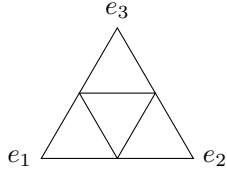
The action of $G/K=\langle \Sbar, \Rbar \rangle \cong S_3$ on $Y_K$ is determined by its action on the toric picture as a dihedral group.
In particular, on the affine open set $U=\Spec \bC[\frac{x_1x_2}{x_3}, \frac{x_2x_3}{x_1}, \frac{x_3x_1}{x_2}]$ corresponding to the central three-dimensional cone,
 the action is that of $S_3$ on $\bA^3$ induced by the permutation of coordinates.
 The complement of $U$ is the union of three toric divisors $\Spec \bC[x_1^2, x_2^2] \sqcup \Spec \bC[x_2^2, x_3^2] \sqcup \Spec \bC[x_3^2, x_1^2]$ and a non-trivial stabilizer of a point on the divisors acts locally as a reflection on $Y_K$.
 Therefore, the quotient $Y_K/(G/K)$ is a smooth threefold.
 On the other hand, 
 the singular locus $\Sing(X_K)$ of $X_K$ is the image of the union of three coordinate axes in $\bA^3$.
 Hence, the image of $\Sing(X_K)$ in $X$ is the coordinate axis $L:=\{(t, 0, 0) \} \subset \bA^3$ which is isomorphic to $\bA^1$.
 The exceptional locus $\Exc(f_K)$ of the crepant resolution $f_K:Y_K \to X_K$ is the union of three toric divisors corresponding to
 the lattice points $\frac{e_1+e_2}{2}$, $\frac{e_2+e_3}{2}$, $\frac{e_3+e_1}{2}$ in Figure \ref{fig:toric}.
One can see that the image $Z$ of $\Exc(f_K)$ in $Y_K/(G/K)=Y_K/S_3$ is a smooth surface,
its image in $X$ is $L$, and a fiber of $Z \to L$ is $\bP^1$. 
This implies that $Y_K/(G/K)$ is the blowup of $X$ along $L$.

 Notice that the action of $G/K\cong S_3$ on $U$ is the product of the trivial action on the fixed point locus $Y_K^{(G/K)}\cong \bA^1$ and the action of the two-dimensional dihedral case for $n=3$ in Example \ref{ex:dihedral}. Then, we find that the quotient $Y_K/(H/K)$ by $H/K\cong \bZ/3\bZ$ has a family of $A_2$-singularity parametrized by $\bA^1$. Moreover, it follows that the image $L'$ of $Y_K^{(G/K)}$ in the smooth threefold $Y_K/(G/K)$ is still isomorphic to $\bA^1$, and that $Y$ is the blowup of $Y_K/(G/K)$ along $L'$. Hence, $Y$ is obtained from $X$ by blowing up twice along $\bA^1$ and we have
 \begin{equation}\label{eq:tetrahedral_1}
 D(Y) \cong \langle  D(\bA^1), D(\bA^1), D(X)\rangle.
 \end{equation}

Finally, we describe $\Dtilde$. 
Since $G$ has only one conjugacy class of reflections, $D$ is irreducible. 
We can see that the fixed point locus $Y_K^{(G/H)}$ is isomorphic to a one-point blowup of $\bA^2$, where the projective line corresponds to the two-dimensional upper face of the central three-dimensional cone in Figure  \ref{fig:toric}. Since $Y_K^{(G/H)}$ contains the affine line $Y_K^{(G/K)}$, the image of $Y_K^{(G/H)}$ in $Y_K/(G/K)$ has a family of ordinary cusp along $L'$. 
Hence, the blowup of it along $L'$ is again a one-point blowup of $\bA^2$.
 Thus we obtain
 \begin{equation}\label{eq:tetrahedral_2}
 D^{G/H}(Y_H) \cong \langle D(\bA^0), D(\bA^2), D(Y) \rangle.
 \end{equation}
Combining \eqref{eq:tetrahedral_0}, \eqref{eq:tetrahedral_1} and \eqref{eq:tetrahedral_2},
we have
\[
D^G(\bA^3) \cong \langle D(\bA^0), D(\bA^2), D(\bA^1), D(\bA^1), D(X)\rangle.
\]
Since the number of conjugacy classes in $G$ whose fixed point locus is of dimension $i$ is
$1, 2, 1, 1$ for $i=0,1,2,3$, respectively, this proves Theorem \ref{thm:sod} for the tetrahedral group.

\begin{remark}
The crepant resolution $Y_H$ of $\bA^3/H$ constructed above coincides with
the one constructed by Bertin and Markushevich in \cite[Th\'{e}or\`{e}me 3.1]{MR1273078}.
Indeed, \cite{MR1273078} constructs the crepant resolution as a double covering of $Y$ as above.
\end{remark}

\begin{remark}\label{rem:tetra}
The $H$-Hilbert scheme $H\hilb(\bA^3)$ is studied in \cite{GNS1}, where
the fiber of $H\hilb(\bA^3) \to X_H$ over the origin is shown to be the chain of three copies of $\bP^1$ (\cite[2.5]{GNS1}).
One can see that there are two fixed points of the action of a reflection on the fiber,  one  is contained
in the strict transform of $D$, and the other is isolated.
Thus, $H\hilb(\bA^3)/(G/H)$ is singular also in this case.
\end{remark}

\subsection{The case of the octahedral group}\label{subsec:octahedral}
Let $G\subset \GL(3,\bC)$ be the subgroup consisting of elements that preserve the octahedron in $\bR^3$ and define $H:=G\cap \SL(3,\bC)$. 
Then, $G$ is decomposed as $H\times \{ \pm 1\}$ and $H$ is isomorphic to $S_4$. 
We put
\begin{align*}
X_H&:=\bA^3/H, &X&:=\bA^3/G\cong \bA^3.
\end{align*}
Since there are two conjugacy classes of reflections in $G$, the branch divisor $D$ of the double covering $X_H\to X$ has two irreducible components: one, denoted by $D_1$, is singular, while the other, denoted by $D_2$, is isomorphic to $\bA^2$. The singular locus of $D$ consists of three affine lines $L_i$ which are the image of three rotation axes of order $i=2,3,4$. The singular locus of $D_1$ is precisely $L_3$, and the intersection of $D_1$ and $D_2$ consists of $L_2$ and $L_4$.

We briefly recall the construction of a crepant resolution of $X_H$ by Bertin and Markushevich \cite[Th\'{e}or\`{e}me 3.1 \'{E}tape 2]{MR1273078} in the below. First, we obtain a map $f: Y\to X$ by successively blowing up four affine lines contained in the singular locus of $D$ or in the singular locus of the strict transform of $D$. Since one can see the strict transform $\Dtilde=\Dtilde_1\cup \Dtilde_2$ of $D$ under $f$ is smooth, the double covering $Y_H$ of $Y$ branched along $\Dtilde$ is smooth, too. Moreover, the following equation is verified in \cite{MR1273078}:
\begin{equation}\label{eq:octahedral_logcrepant}
    K_Y+\frac{1}{2}\Dtilde=f^*(K_X+\frac{1}{2}D).
    \end{equation}
Thus, the induced projective birational morphism $f_H: Y_H\to X_H$ is a crepant resolution, 
and we can conclude that the composite of four blowups $f:Y\to X$ is a maximal resolution of $(X,\frac{1}{2}D)$ by Lemma \ref{lem:max}.

Now, by applying Proposition \ref{prop:yamagishi} to $f_H$, we have an equivalence
 \begin{equation}\label{eq:octahedral_1}
 D^G(\bA^3)\cong D^{G/H}(Y_H).
 \end{equation}
 Therefore, the construction above establishes Theorem \ref{thm:equivalence} (1)(2), leading to the following decompositions:
\begin{equation}\label{eq:octahedral_2}
 D^{G/H}(Y_H) \cong \langle D(\Dtilde_1), D(\Dtilde_2), D(Y) \rangle,
 \end{equation}
\begin{equation}\label{eq:octahedral_3}
 D(Y) \cong \langle  D(\bA^1), D(\bA^1), D(\bA^1), D(\bA^1),D(X)\rangle.
 \end{equation}
Since the stack $\scY$ in Theorem \ref{thm:equivalence} (4) is the quotient stack $[Y_H/(G/H)]$, the equivalence \eqref{eq:octahedral_1} confirms Theorem \ref{thm:equivalence} (4).

To prove Theorem \ref{thm:equivalence} (3), we need a more explicit description of $\Dtilde$. By examining the equations in \cite[Th\'{e}or\`{e}me 3.1 \'{E}tape 2]{MR1273078} (along with additional calculations for two blowups left to the reader in \cite{MR1273078}), we find that among the four projective lines that form the fiber $f^{-1}(0)$, three are contained in $\Dtilde_1$ while the remaining one is not in $\Dtilde$. Away from the origin, the fibers are finite in $\Dtilde$. This observation implies that $\Dtilde_2$ is isomorphic to $\bA^2$ and $\Dtilde_1$ is obtained by three successive blowups at points in $\bA^2$. Indeed, since $D_2$ is $\bA^2$ and the normalization of $D_1$ is isomorphic to $\bA^2$ due to the existence of a finite birational morphism $\bA^2\to D_1$ \cite[Proposition 2.2.2]{PVdB}, the induced morphism from $\Dtilde_1$ to $\bA^2$ is a sequence of blowdowns of the projective lines over the origin in $X$. Thus, the derived category $D(\Dtilde_1)$ decomposes as 
\begin{equation}\label{eq:octahedral_4}
 D(\Dtilde_1)\cong \langle D(\bA^0), D(\bA^0), D(\bA^0), D(\bA^2)\rangle.
 \end{equation}

Combining \eqref{eq:octahedral_1},  \eqref{eq:octahedral_2}, \eqref{eq:octahedral_3}, and \eqref{eq:octahedral_4}, 
we see that $D^G(\bA^3)$ decomposes into three copies of $D(\bA^0)$, four copies of $D(\bA^1)$, two copies of $D(\bA^1)$, and $D(X)$.
Since $G$ has $3, 4, 2, 1$ conjugacy classes with fixed point loci of dimension $0,1,2,3$, respectively, this confirms Theorem \ref{thm:sod} for the octahedral group.

\begin{remark}\label{rem:octa}
A description of the fiber $H\hilb(\bA^3) \to X_H$ is given in \cite[3.6]{GNS2}.
It consists of four curves, three of which meet transversally at one point.
One can see that this point is isolated in the fixed point set of the action of $-1$
and $H\hilb(\bA^3)/(G/H)$ is again singular.
\end{remark}
\subsection{The case of the icosahedral group}\label{subsec:icosahedral}
Let $G\subset \GL(3,\bC)$ be the subgroup consisting of elements that preserve the icosahedron in $\bR^3$, and define $H:=G\cap \SL(3,\bC)$.
$G$ is isomorphic to $H\times \{\pm 1\}$ and $H$ is isomorphic to $A_5$. 
We define
\begin{align*}
X_H&:=\bA^3/H, &X&:=\bA^3/G\cong \bA^3.
\end{align*}
The double cover $X_H\to X$ is branched along the discriminant divisor $D\subset X\cong \bA^3$. 
Since there is only one conjugacy class of reflections, $D$ is irreducible and its singular locus is given by the images of three rotation axes $L_i$ of $G$, where $L_i$ is the image of the axis of an order-$i$ rotation for $i=2,3,5$. One can easily see that each $L_i$ is isomorphic to $\bA^1$. 

Following Roan \cite[\S 4]{MR1284568} (which adopts the strategy in \cite{MR1273078}), we obtain a crepant resolution $f_H: Y_H\to X_H$. Here, $Y_H$ is constructed as the double covering branched along a smooth divisor $\Dtilde$ of a smooth threefold $Y$. $Y$ results from a sequence of blowups of $X$ along four affine lines and $\Dtilde$ is the strict transform of $D$. As in \S\ref{subsec:octahedral}, one can see that the composition of four blowups $f: Y\to X$ is a maximal resolution of the pair $(X,\frac{1}{2}D)$. By applying Proposition \ref{prop:yamagishi} to $f_H$, we obtain an equivalence
\begin{equation}\label{eq:icosahedral_1}
 D^{G/H}(Y_H)\cong D^G(\bA^3).
 \end{equation}
 Since the stack $\scY$ in Theorem \ref{thm:equivalence} (4) is the quotient stack $[Y_H/(G/H)]$, the equivalence \eqref{eq:icosahedral_1} confirms Theorem \ref{thm:equivalence} (4). Moreover, since $\pi_{G/H}: Y_H\to Y$ is the double covering branched along the smooth divisor $\Dtilde$, we have
\begin{equation}\label{eq:icosahedral_2}
 D^{G/H}(Y_H) \cong \langle D(\Dtilde), D(Y)\rangle.
 \end{equation}
 The fact that $f:Y\to X$ is a sequence of four blowups along $\bA^1$ implies 
\begin{equation}\label{eq:icosahedral_3}
 D(Y) \cong \langle  D(\bA^1), D(\bA^1), D(\bA^1), D(\bA^1), D(X)\rangle.
 \end{equation}
Hence, Theorem \ref{thm:kawamata} (1)(2) are now established.
 
We describe $\Dtilde$ more explicitly for Theorem \ref{thm:equivalence} (3). By analyzing the equations in \cite[\S 4]{MR1284568}, one can see that all of the four projective lines that form the fiber of $f$ over the origin in $X$ are contained in $\Dtilde$, whereas the fibers over the other points of $D$ are finite in $\Dtilde$. Hence, as in \S \ref{subsec:octahedral}, it follows that $\Dtilde$ is obtained from $\bA^2$ through four successive blowups at points. Consequently, Theorem \ref{thm:equivalence}(3) is verified, and we obtain a decomposition 
\begin{equation}\label{eq:icosahedral_4}
 D(\Dtilde)\cong \langle  D(\bA^0), D(\bA^0), D(\bA^0), D(\bA^0), D(\bA^2)\rangle.
 \end{equation}
 
Combining \eqref{eq:icosahedral_1}, \eqref{eq:icosahedral_2},
\eqref{eq:icosahedral_3}, and \eqref{eq:icosahedral_4}, we
conclude that $D^G(\bA^3)$ decomposes into four copies of $D(\bA^0)$,
four copies of $D(\bA^1)$, $D(\bA^2)$ and $D(X)$.
As a consequence, Theorem \ref{thm:sod} for the icosahedral group follows from the fact that the conjugacy classes in $G$ with fixed point loci of dimension $0,1,2$ and $3$ are counted as 
$4, 4, 1$ and $1$, respectively.
\begin{remark}\label{rem:icosa}
The fiber of $H\hilb(\bA^3) \to X_H$ consists of four curves, three of which meet at a point  \cite[Corollary 3.6]{GNS1}. The point is again isolated in the fixed point locus of the action of $-1$ and $H\hilb(\bA^3)/(G/H)$ is singular.
\end{remark}
\bibliographystyle{abbrv}
\def\cprime{$'$} \def\cprime{$'$}

\end{document}